\documentclass{amsart}
\usepackage{amsmath,amssymb,amsfonts}
\usepackage[mathscr]{eucal}

\usepackage{enumerate}

\theoremstyle{plain}
\newtheorem{theorem}{Theorem}[section]

\newtheorem{prop}[theorem]{Proposition}
\newtheorem{lemma}[theorem]{Lemma}

\theoremstyle{definition}
\newtheorem{definition}[theorem]{Definition}
\newtheorem{example}[theorem]{Example}

\input xy
\xyoption{all}

\newcommand{\Deltaop}{{\bf \Delta}^{op}}

\newcommand{\nerve}{\text{nerve}}

\newcommand{\we}{\text{we}}

\newcommand{\Hom}{\text{Hom}}
\newcommand{\Map}{\text{Map}}
\newcommand{\Aut}{\text{Aut}}
\newcommand{\Ho}{\text{Ho}}
\newcommand{\Ext}{\text{Ext}}
\newcommand{\SSets}{\mathcal{SS}ets}

\newcommand{\Sets}{\mathcal Sets}

\newcommand{\map}{\text{map}}
\newcommand{\ob}{\text{ob}}
\newcommand{\hoequiv}{\text{hoequiv}}

\newcommand{\css}{\mathcal{CSS}}

\newcommand{\dhw}{\mathcal{DH}(W)}
\newcommand{\id}{\text{id}}

\begin{document}

\title[Derived Hall algebras]{Derived Hall algebras for stable homotopy theories}

\author[J.E. Bergner]{Julia E. Bergner}

\address{Department of Mathematics, University of California, Riverside, CA 92521}

\email{bergnerj@member.ams.org}

\date{\today}



\thanks{The author was partially supported by NSF grant DMS-0805951.  Support from the Fields Institute in Spring 2007 and the CRM Barcelona in Spring 2008, where some of this work was completed, is also gratefully acknowledged.}

\begin{abstract}
In this paper we extend To\"en's derived Hall algebra construction, in which he obtains unital associative algebras from certain stable model categories, to one in which such algebras are obtained from more general stable homotopy theories, in particular stable complete Segal spaces satisfying appropriate finiteness assumptions.
\end{abstract}

\maketitle

\section{Introduction}

Hall algebras associated to abelian categories play an important role in representation theory.  In particular, when the abelian category in question is the category of $\mathbb F_q$-representations of a quiver associated to a simply-laced Dynkin diagram, there is a close relationship between the Hall algebra and the quantum enveloping algebra of the Lie algebra associated to the same Dynkin diagram. Recent attempts to strengthen this relationship have led to the problem of associating some kind of Hall algebra to categories which are triangulated rather than abelian.  In particular, it is conjectured that one could recover the quantum enveloping algebra from an appropriate Hall-type algebra associated to Peng and Xiao's root category, which is, roughly speaking, the derived category of the category of this abelian category of representations, modulo a double shift relation \cite{px}.

In \cite{toendha}, To\"en constructs ``derived Hall algebras" associated to triangulated categories arising as homotopy categories of model categories whose objects are modules over a sufficiently finitary differential graded category over $\mathbb F_q$.  In doing so, he develops a formula for the multiplication in this algebra in such a way that it can be regarded as a generalization of the formula for the multiplication in an ordinary Hall algebra.  This formula was verified for more general triangulated categories, still satisfying certain finiteness conditions, by Xiao and Xu \cite{xx}.  However, none of these methods can yet be applied to the root category, as it does not satisfy these finiteness assumptions.

In this paper, we seek to generalize To\"en's development of derived Hall algebras.  Specifically, we modify his proof to establish derived Hall algebras corresponding to triangulated categories arising as homotopy categories for more general stable homotopy theories. Most triangulated categories can be realized as homotopy categories of such stable homotopy theories.  Although such triangulated categories are covered by Xiao and Xu's work, our objective is rather to broaden the context in which we can make use of homotopy-theoretic methods.  We expect that these ideas will shed light on the question of how to find a similar algebra arising from a triangulated category which is not finitary.  Also, it seems that this more flexible setting should be more amenable than the model category world for finding a coalgebra or even a Hopf algebra structure on derived Hall algebras, extending these structures which are significant in the study of ordinary Hall algebras.  This idea will be the subject of future work in collaboration with Robertson.

We expect that the methods of this paper will be applicable to other settings, enabling one to use more general stable homotopy theories in settings in which the additional structure of stable model categories is too restrictive.  For example, not all derived categories arise from actual model categories, but they do always come from a stable homotopy theory.  It is expected that the ability to work with such homotopy theories, which contain more information than their associated derived categories, will facilitate progress in the many areas in which derived categories appear.

In this paper, we use the complete Segal space model for homotopy theories.  If we regard a homotopy theory as a category with weak equivalences, then there are several equivalent models for homotopy theories as mathematical objects, in particular objects of model categories with appropriate weak equivalences.  Complete Segal spaces were developed by Rezk \cite{rezk}; they are simplicial spaces satisfying conditions enabling one to regard them as something like a simplicial category up to homotopy.  Their associated model category is in fact equivalent to the model structure on the category of simplicial categories \cite{thesis}, as well as to the model structures for Segal categories \cite{thesis} and quasi-categories \cite{jt}.  While any one of these models could be used, we prefer the complete Segal space model here because it is particularly well-suited for understanding fiber products of model categories \cite{fiberprod}, one of the key tools used by To\"en in his proof of the associativity of derived Hall algebras.  Specifically, we are able to use homotopy pullbacks of complete Segal spaces where he used the homotopy fiber product of model categories.

There is, in fact, another perspective on complete Segal spaces (and equivalent objects); they are also models for $(\infty,1)$-categories, or $\infty$-categories with $n$-morphisms invertible for $n>1$.  While the motivation for using complete Segal spaces in this paper arises from the viewpoint that they are generalizations of model categories, it is also useful, in particular when we need to define categorical notions such as colimits within them, to remember that they can be thought of as generalizations of ordinary categories in this way.

In Section 2, we give a review of stable model categories.  These ideas are generalized in Section 3, where we explain how Lurie's methods for stable quasi-categories can be translated to stable complete Segal spaces.  We review our main tool of interest, homotopy fiber products of model categories and homotopy pullbacks of complete Segal spaces, in Section 4, then introduce To\"en's derived Hall algebras in Section 5.  The main results of the paper can be found in Section 6, where we establish derived Hall algebras for stable complete Segal spaces.

%\begin{thank}
%The ideas in this paper have benefitted enormously from conversations with many people over the last few years.  These people include Zongzhu Lin, Bertrand To\"en, Clark Barwick, Mark Behrens, David Gepner, Gon\c{c}alo Tabuada, John Baez, Christopher Walker, and Aviv Censor.
%\end{thank}

\section{Stable model categories}

Recall that a \emph{model category} $\mathcal M$ is a category with three distinguished classes of morphisms: weak equivalences, fibrations, and cofibrations, satisfying five axioms \cite[3.3]{ds}.  Given a model category structure, one can pass to the \emph{homotopy category} $\Ho(\mathcal M)$, which is a localization of $\mathcal M$ with respect to the class of weak equivalences \cite[1.2.1]{hovey}.  In particular, the weak equivalences, as the morphisms that we wish to invert, make up the most important part of a model category.  An object $x$ in a model category $\mathcal M$ is \emph{fibrant} if the unique map $x \rightarrow \ast$ to the terminal object is a fibration.  Dually, an object $x$ in $\mathcal M$ is \emph{cofibrant} if the unique map $\phi \rightarrow x$ from the initial object is a cofibration.

%Given a model category $\mathcal M$, there is also a model structure on the category $\mathcal M^{[1]}$, often called the morphism category of $\mathcal M$.  The objects of $\mathcal M^{[1]}$ are morphisms of $\mathcal M$, and the morphisms of $\mathcal M^{[1]}$ are given by pairs of morphisms making the appropriate square diagram commute.  A morphism in $\mathcal M^{[1]}$ is a weak equivalence (or cofibration) if its component maps are weak equivalences (or cofibrations) in $\mathcal M$.  More generally, $\mathcal M^{[n]}$ is the category with objects strings of $n$ composable morphisms in $\mathcal M$; the model structure can be defined analogously.

The standard notion of equivalence of model categories is given by
the following definitions.  First, recall that an \emph{adjoint
pair} of functors $F \colon \mathcal C \leftrightarrows \mathcal D
\colon G$ satisfies the property that, for any objects $X$ of
$\mathcal C$ and $Y$ of $\mathcal D$, there is a natural
isomorphism
\[ \varphi: \Hom_\mathcal D(FX, Y) \rightarrow \Hom_\mathcal C(X,
GY). \]  The functor $F$ is called the \emph{left adjoint} and $G$
the \emph{right adjoint} \cite[IV.1]{macl}.

\begin{definition} \cite[1.3.1]{hovey}
An adjoint pair of functors $F \colon \mathcal M \leftrightarrows
\mathcal N \colon G$ between model categories is a \emph{Quillen
pair} if $F$ preserves cofibrations and $G$ preserves fibrations.  The left adjoint $F$ is called a \emph{left Quillen functor}, and the right adjoint $G$ is called the \emph{right Quillen functor}.
\end{definition}

\begin{definition} \cite[1.3.12]{hovey}
A Quillen pair of model categories is a \emph{Quillen equivalence}
if for all cofibrant $X$ in $\mathcal M$ and fibrant $Y$ in
$\mathcal N$, a map $f \colon FX \rightarrow Y$ is a weak
equivalence in $\mathcal D$ if and only if the map $\varphi f
\colon X \rightarrow GY$ is a weak equivalence in $\mathcal M$.
\end{definition}

We also consider model categories with the additional data that their homotopy categories are triangulated.  Recall that a \emph{triangulated category} $T$ is an additive category, together with an equivalence $\Sigma \colon T \rightarrow T$ called a \emph{shift functor}, and a collection of \emph{distinguished triangles}
\[ \xymatrix{x \ar[r]^\alpha & y \ar[r]^\beta & z \ar[r]^\gamma & \Sigma x} \] satisfying four axioms \cite[\S 2.1]{krause}.

For a model category to have a triangulated homotopy category, it must first be \emph{pointed}, in that its initial and terminal objects coincide.  Such an object is called a \emph{zero object}.

\begin{definition} \cite[7.1.1]{hovey}
A pointed model category $\mathcal M$ is \emph{stable} if its homotopy category $\Ho(\mathcal M)$ is triangulated.
\end{definition}

\begin{example}
Let $\mathcal R$ be a ring and $Ch(R)$ the category of chain complexes of $R$-modules.  Then the model category structure on $Ch(R)$ is triangulated.  In fact, its homotopy category is equivalent to the \emph{derived category} $\mathcal D(R)$, formed by taking $Ch(R)$ modulo the equivalence relation given by chain homotopies of maps, and formally inverting the quasi-isomorphisms \cite[\S 1.2]{krause}.
\end{example}

\section{Stable complete Segal spaces}

\subsection{Simplicial spaces and complete Segal spaces}

Recall that the simplicial indexing category $\Deltaop$ is defined to be the category with objects finite ordered sets $[n]=\{0 \rightarrow 1 \rightarrow \cdots \rightarrow n\}$ and morphisms the opposites of the order-preserving maps between them.  A \emph{simplicial set} is then a functor
\[ K \colon \Deltaop \rightarrow \Sets. \]
We denote by $\SSets$ the category of simplicial sets, and this category has a natural model category structure equivalent to the standard model structure on topological spaces \cite[I.10]{gj}.

One can consider more general simplicial objects; in this paper we work with \emph{simplicial spaces} (also called bisimplicial sets), or functors
\[ X \colon \Deltaop \rightarrow \SSets. \]
Given a simplicial set $K$, we also denote by $K$ the simplicial space which has the simplicial set $K$ at every level.  We denote by $K^t$, or ``$K$-transposed", the constant simplicial space in the other direction, where $(K^t)_n = K_n$, where on the right-hand side $K_n$ is regarded as a discrete simplicial set.  The category of simplicial spaces has a model category structure called the \emph{Reedy structure} in which weak equivalences are given levelwise and all objects are cofibrant \cite{reedy}.

Specifically, we consider simplicial spaces satisfying additional conditions, namely, those inducing a notion of composition up to homotopy.  These Segal spaces and complete Segal spaces were first introduced by Rezk \cite{rezk}, and the name is meant to be suggestive of similar ideas first presented by Segal \cite{segal}.

\begin{definition} \cite[4.1]{rezk}
A \emph{Segal space} is a Reedy fibrant simplicial space $W$ such that the Segal maps
\[ \varphi_n \colon W_n \rightarrow \underbrace{W_1 \times_{W_0} \cdots \times_{W_0} W_1}_n \] are weak equivalences of simplicial sets for all $n \geq 2$.
\end{definition}

Given a Segal space $W$, we can consider its \emph{objects} $\ob(W)= W_{0,0}$, and, between any two objects $x$ and $y$, the \emph{mapping space} $\map_W(x,y)$, given by the homotopy fiber of the map $W_1 \rightarrow W_0 \times W_0$ given by the two face maps $W_1 \rightarrow W_0$.  The Segal condition stated above guarantees that a Segal space has a notion of $n$-fold composition of mapping spaces, up to homotopy.

The \emph{homotopy category} of $W$, denoted $\Ho(W)$, has as objects the elements of the set $W_{0,0}$, and \[ \Hom_{\Ho(W)}(x,y) = \pi_0 \map_W(x,y). \]  A \emph{homotopy equivalence} in $W$ is a 0-simplex of $W_1$ whose image in $\Ho(W)$ is an isomorphism.  We consider the subspace of $W_1$ whose components contain homotopy equivalences, denoted $W_{\hoequiv}$.  Notice that the degeneracy map $s_0 \colon W_0 \rightarrow W_1$ factors through $W_{\hoequiv}$; hence we may make the following definition.

\begin{definition} \cite[\S 6]{rezk}
A \emph{complete Segal space} is a Segal space $W$ such that the map $W_0 \rightarrow W_{\hoequiv}$ is a weak equivalence of simplicial sets.
\end{definition}

Given this definition, we can describe the complete Segal space model structure on the category of simplicial spaces.

\begin{theorem} \cite[7.2]{rezk}
There is a model structure $\css$ on the category of simplicial spaces such that the fibrant and cofibrant objects are precisely the complete Segal spaces. Furthermore, $\css$ has the additional structure of a cartesian closed model category.
\end{theorem}

The fact that $\css$ is cartesian closed allows us to consider, for any complete Segal space $W$ and simplicial space $X$, the complete Segal space $W^X$.  In particular, using the simplicial structure, the simplicial set at level $n$ is given by
\[ (W^X)_n = \Map(X \times \Delta[n]^t, W). \]

\subsection{Stable quasi-categories and stable complete Segal spaces}

As with model categories, we need to consider complete Segal spaces which are \emph{stable}, in the sense that their homotopy categories are triangulated.  It should be noted that, although we have given this simple definition of a stable complete Segal space, one could define it in a more technical way which permits a better understanding of the structure of a stable complete Segal space; Lurie has explained these ideas extensively for stable quasi-categories in \cite{luriestable}, and they can fairly easily be translated into the equivalent setting of complete Segal spaces.

Although we do not go into this level of detail on this point in this paper, there are other notions that have been developed for quasi-categories which are useful here for complete Segal spaces.  Thus, we give a very brief summary of quasi-categories and their relationship with complete Segal spaces.

Recall that a \emph{quasi-category} $X$ is a simplicial set satisfying the inner Kan condition, so that for any $n \geq 1$ and $0 <k<n$, a dotted arrow lift exists in any diagram of the form
\[ \xymatrix{V[n,k] \ar[r] \ar[d] & X \\
\Delta[n]. \ar@{-->}[ur] &  } \]  The notion of quasi-category goes back to Boardman and Vogt \cite{bv}, but is has received extensive attention more recently, especially by Joyal \cite{joyal} and Lurie \cite{lurie}.  In particular, Joyal proves that there is a model structure on the category of simplicial sets such that the fibrant and cofibrant objects are precisely the quasi-categories.  We denote this model category $\mathcal{QC}at$.  Furthermore, Joyal and Tierney have proved that the model category $\mathcal{QC}at$ is Quillen equivalent to Rezk's model category $\css$ \cite{jt}.  Remarkably, they prove that there are actually two different Quillen equivalences between these two model categories.  Here, we make use of the one that is particularly easy to describe, the right adjoint $\css \rightarrow \mathcal{QC}at$ given by $W \mapsto W_{\ast, 0}$.

Using this relationship we return to the matter of explaining some necessary structures on complete Segal spaces.  For a complete Segal space to be stable, we need it to be pointed, or to have a zero object, denoted 0.  As we have seen, in an ordinary category, a zero object is one which is both initial and terminal, so for any object $x$, there are unique morphisms $x \rightarrow 0$ and $0 \rightarrow x$.  As a complete Segal space is a homotopical generalization of a category, we require a homotopical notion of initial and terminal objects.  The following definitions, given by Joyal \cite{joyal}  and Lurie \cite[1.2.12.1, 1.2.12.6]{lurie} for quasi-categories, are easy to reformulate for complete Segal spaces.

\begin{definition}
An object $x \in W_{0,0}$ of a complete Segal space is \emph{initial} if it is initial as an object of $\Ho(W)$, i.e., if $\map_W(x,y)$ is weakly contractible for any $y \in W_{0,0}$.  Dually, $x$ is \emph{terminal} if it is terminal as an object of $\Ho(W)$, i.e., if $\map_W(y,x)$ is weakly contractible for any $y$.  An object is a \emph{zero object} of $W$ if it is both initial and terminal.
\end{definition}

In addition to having a zero object, we need to have a notion of ``pushout" within a complete Segal space, another analogue of a standard categorical idea within this generalized setting.  Fortunately, formal definitions of limits and colimits within quasi-categories have been established by Lurie \cite[1.2.13.4]{lurie}.  We give a brief exposition here, enough to translate his definition into the world of complete Segal spaces; see \cite[1.2.8, 1.2.13]{lurie} for a detailed treatment.

Let $X$ and $Y$ be simplicial sets.  We can define their \emph{join} $X \star Y$ by
\[ (X \star Y)_n = X_n \amalg Y_n \amalg \coprod_{i+j=n-1} X_i \times Y_j. \]  Note that the operation defines a monoidal product on $\SSets$ with unit the empty simplicial set $\phi$.  Then, for a fixed simplicial set $X$, we can define a functor
\[ X \star (-) \colon \SSets \rightarrow \SSets \] by
\[ Y \mapsto X \star Y \] and notice that the map $\phi \rightarrow Y$ is sent to the map $X \star \phi =X \rightarrow X \star Y$.  Thus, the simplicial set $X \star Y$ comes equipped with a canonical map $X \rightarrow X \star Y$, and so we can regard $X \star Y$ as an object of the \emph{undercategory} or \emph{category of simplicial sets under} $X$ \cite[II.6]{macl}, denoted $X \downarrow \SSets$.  In doing so, we can think of our functor as
\[ X \star (-) \colon \SSets \rightarrow X \downarrow \SSets. \]  This functor has a right adjoint given by
\[ (p \colon X \rightarrow Y) \mapsto Y. \]  To remember that $Y$ has come from some map $p \colon X \rightarrow Y$, Lurie denotes the image of this functor $Y_{p/}$.  We can think of $Y_{p/}$ as the simplicial set $Y$ with a specified $X$-shaped diagram inside it.

Such an object can be used to define colimits in a quasi-category.  If $Y$ is a quasi-category and $p \colon X \rightarrow Y$ is a map of simplicial sets, then a \emph{colimit} for $p$ is an initial object of $Y_{p/}$.  Dually, one could use the functor $(-) \star X$, its right adjoint, and the resulting definition of $Y_{/p}$ to define a limit in a quasi-category $Y$.

Now, we translate this definition into $\css$.

\begin{definition}
Let $W$ be a complete Segal space and $X$ a simplicial set, together with a map $p \colon X^t \rightarrow W$.  A \emph{colimit} for $p$ in $W$ is an initial object of $(W_{\ast,0})_{p/}$, regarded as an object of $W$.
\end{definition}

In this paper, we consider the case where the simplicial set $X$ is $\Delta[1] \amalg_{\Delta[0]} \Delta[1]$, forming the diagram $\cdotp \leftarrow \cdotp \rightarrow \cdotp$, so that the colimit is a ``pushout" in the complete Segal space $W$.  One can show that if $W$ is stable, the fact that $\Ho(W)$ is triangulated guarantees that colimits must always exist in $W$.  Again, we refer the reader to Lurie's manuscript on stable quasi-categories \cite{luriestable} for greater depth on this point.

\subsection{Model categories and complete Segal spaces}

We conclude this section with a brief exposition on the relationship between model categories and complete Segal spaces.  Since we are translating a construction on model categories to one on complete Segal spaces, we need to understand how to regard a model category as a specific kind of complete Segal space.

As described by Rezk \cite{rezk}, any category with weak equivalences gives rise to a complete Segal space via the functor we denote $L_C$; given such a category $\mathcal C$, $L_C \mathcal C$ is given by
\[ (L_C \mathcal C)_n = \nerve(\we (\mathcal C^{[n]})) \] where $\we(\mathcal C^{[n]})$ denotes the category of weak equivalences of chains of $n$ composable morphisms in $\mathcal C$.

If $\mathcal M$ is a model category, then we can apply this construction, but, as explained in \cite{fiberprod}, it is only a functor when the morphisms between model categories preserve weak equivalences.  Since we want a construction which is functorial on the category of model categories with left Quillen functors between them, we can modify the construction by restricting to the full subcategory of $\mathcal M$ whose objects are cofibrant.

The main result of \cite{css} is that this construction is well-behaved with respect to other natural ways of getting a complete Segal space from a model category; in particular, the resulting complete Segal space is weakly equivalent to the one obtained from taking the simplicial localization and then applying any one of several functors from simplicial categories to complete Segal spaces. There is an up-to-homotopy characterization of the resulting complete Segal space as well.  While we do not make use of this description explicitly in this paper, it is key to the proof of Theorem \ref{Fiber} below.

\section{Fiber products of model categories and homotopy pullbacks of complete Segal spaces}

A key tool in To\"en's proof that his derived Hall algebras are associative is the fiber product of model categories.  We begin with his definition as given in \cite{toendha}.
First, suppose that
\[ \xymatrix@1{\mathcal M_1 \ar[r]^{F_1} & \mathcal M_3 & \mathcal
M_2 \ar[l]_{F_2}} \] is a diagram of left Quillen functors of
model categories.  Define their \emph{fiber product} to be the
model category $\mathcal M = \mathcal M_1 \times^h_{\mathcal M_3}
\mathcal M_2$ whose objects are given by 5-tuples $(x_1, x_2, x_3;
u, v)$ such that each $x_i$ is an object of $\mathcal M_i$ fitting
into a diagram
\[ \xymatrix@1{F_1(x_1) \ar[r]^-u & x_3 & F_2(x_2) \ar[l]_v.} \]
A morphism of $\mathcal M$, say $f \colon (x_1, x_2, x_3; u, v)
\rightarrow (y_1, y_2, y_3; z, w)$, is given by maps $f_i: x_i
\rightarrow y_i$ such that the following diagram commutes:
\[ \xymatrix{F_1(x_1) \ar[r]^-u \ar[d]^{F_1(f_1)} & x_3 \ar[d]^{f_3} &
F_2(x_2) \ar[l]_-v \ar[d]^{F_2(f_2)} \\
F_1(y_1) \ar[r]^-z & y_3 & F_2(y_2) \ar[l]_-w.} \]

This category $\mathcal M$ can be given the structure of a model
category, where the weak equivalences and cofibrations are given
levelwise.  In other words, $f$ is a weak equivalence (or
cofibration) if each map $f_i$ is a weak equivalence (or
cofibration) in $\mathcal M_i$.

A more restricted definition of this construction requires that
the maps $u$ and $v$ be weak equivalences in $\mathcal M_3$.  Unfortunately, if we impose this additional condition, the resulting category cannot be given the structure of a model category because it does not have sufficient limits and colimits.  However, it is still a perfectly good category with weak equivalences, and in some cases we can localize $\mathcal M$ so that the fibrant-cofibrant objects of the localized model category have $u$ and $v$ weak equivalences \cite{fiberprod}.  Although To\"en uses the model structure given above, at the point where he really makes use of the fiber product he restricts to the case where the maps $u$ and $v$ are weak equivalences.  Thus, we assume here this extra structure.

Consider the functor $L_C$, described in the previous section, which takes a model category (or
category with weak equivalences) to a complete Segal space.  Given a
fiber square of model categories where we require the maps $u$ and $v$ to be weak equivalences, we can apply this functor to obtain a commutative square
\[ \xymatrix{L_C \mathcal M \ar[r] \ar[d] & L_C \mathcal M_2
\ar[d] \\
L_C \mathcal M_1 \ar[r] & L_C \mathcal M_3.} \]

Alternatively, we could apply the functor $L_C$ only to the
original diagram and take the homotopy pullback, which we denote
$P$, and obtain the following diagram:
\[ \xymatrix{P \ar[r] \ar[d] & L_C \mathcal M_2
\ar[d] \\
L_C \mathcal M_1 \ar[r] & L_C \mathcal M_3.} \]

\begin{theorem} \cite{fiberprod} \label{Fiber}
The complete Segal spaces $L_C \mathcal M$ and $P=L_C\mathcal M_1 \times^h_{L_C\mathcal M_3} \L_C \mathcal M_2$ are weakly equivalent.
\end{theorem}

This theorem allows us to use the homotopy pullback of complete Segal spaces to generalize the situations in which To\"en uses the fiber product of model categories.  In particular, we generalize a scenario given by To\"en \cite[4.2]{toendha} as follows.

Let
\[ \xymatrix{W \ar[r]^{H_1} \ar[d]_{H_2} & X \ar[d]^{F_1} \\
Y \ar[r]^{F_2} & Z} \] be diagram of complete Segal spaces equipped with an isomorphism $\alpha \colon F_1 \circ H_1 \Rightarrow F_2 \circ H_2$, and define a map
\[ F \colon W \rightarrow V=X \times^h_{Z} Y \] by
\[ w \mapsto (H_1(w), H_2(w); \alpha_w). \]

\begin{lemma} \label{lemma}
If $\Ho(W) \rightarrow \Ho(V)$ is an equivalence of categories, then the diagram
\[ \xymatrix{\nerve(\Ho(wW)) \ar[r] \ar[d] & \nerve(\Ho(wX)) \ar[d] \\
\nerve(\Ho(wY)) \ar[r] & \nerve(\Ho(w Y))} \] is homotopy cartesian.
\end{lemma}

\begin{proof}
We want to show that the map
\[ \nerve(\Ho(wW)) \rightarrow \nerve(\Ho(wX)) \times^h_{\nerve(\Ho(wZ))} \nerve(\Ho(wY)) \] is a weak equivalence of simplicial sets.  By our assumption, we know that the map
\[ \Ho(W) \rightarrow \Ho(X \times^h_{Z} Y) \] is an equivalence of categories.  Notice that the homotopy category $\Ho(wW)$ is the maximal subgroupoid of $\Ho(W)$ and analogously for the other complete Segal spaces in the diagram.  Hence, we have an equivalence of categories
\begin{multline*}
\Ho(wW) \rightarrow \Ho(w(X \times^h_{Z} Y)) \simeq \Ho(wX \times^h_{wZ} wY) \\ \simeq \Ho(wX) \times^h_{\Ho(wZ)} \Ho(wY).
\end{multline*}
Since nerves of equivalent categories are weakly equivalent simplicial sets, the lemma follows.
\end{proof}

\section{Hall algebras and derived Hall algebras}

\subsection{Classical Hall algebras}

Let $\mathcal A$ be an abelian category.  Throughout this section, we assume that $\mathcal A$ is \emph{finitary}, in that, for any objects $x$ and $y$ of $\mathcal A$, the groups $\Hom(x,y)$ and $\Ext^1(x,y)$ are finite.
%Throughout this section, we assume that $\mathcal A$ has only finitely many isomorphism classes of objects, and hence, in particular, for any pair of objects $x$ and $y$ in $\mathcal A$, the cardinality of $\Ext^1(x,y)$ is finite.

\begin{definition} \cite{schiff}
Given an abelian category $\mathcal A$, its \emph{Hall algebra} $\mathcal H(\mathcal A)$ is defined as
\begin{enumerate}
\item the vector space with basis isomorphism classes of objects in $\mathcal A$, with

\item multiplication given by
\[ [x] \cdotp [y]=\sum_{[z]} g^z_{x,y} [z] \] where the \emph{Hall numbers} $g^z_{x,y}$ are given by
\[ g^z_{x,y}= \frac{|\{0 \rightarrow x \rightarrow z \rightarrow y \rightarrow 0 \text{ exact} \} |}{|\Aut(x)| \cdotp |\Aut(y)|}. \]
\end{enumerate}
\end{definition}

Notice that our assumptions on $\mathcal A$ guarantee that each Hall number really is a finite number.  It can be shown that this definition gives $\mathcal H(\mathcal A)$ the structure of a unital associative algebra \cite{ringel}.

Although Hall algebras have been investigated for a number of purposes, recent interest in them has arisen from the close relationship between Hall algebras and quantum groups in the following situation.  Suppose that $\mathfrak g$ is a Lie algebra of type $A$, $D$, or $E$.  Then $\mathfrak g$ has an associated simply-laced Dynkin diagram, which is just an unoriented graph with no cycles.  Assigning an orientation to each of the edges in this graph gives a quiver, or oriented graph, which we denote $Q$.  Given a finite field $\mathbb F_q$, let $\mathcal A$ be the category of $\mathbb F_q$-representations of this quiver $Q$.  It can be shown that $\mathcal A$ is in fact an abelian category satisfying our finiteness assumptions, and hence we have an associated Hall algebra $\mathcal H(\mathcal A)$ \cite{ringel}.  The Hall algebra as we have defined it is not independent of the chosen orientation on the quiver, but a slight modification by Ringel makes it so; this algebra is often called the \emph{Ringel-Hall algebra} \cite{ringelrevis}.

However, another algebra can be obtained from $\mathfrak g$, namely the quantum enveloping algebra $U_q(\mathfrak g)$.  This algebra can be given its triangular decomposition
\[ U_q(\mathfrak g) = U_q(\mathfrak n^+) \otimes U_q(\mathfrak h) \otimes U_q(\mathfrak n^-). \]  Work of Ringel, further developed by Green, has shown that there is a close relationship between the Hall algebra $\mathcal H(\mathcal A)$ and the positive part of the quantum enveloping algebra,
\[ U_q(\mathfrak b^+) = U_q(\mathfrak n^+) \otimes U_q(\mathfrak h) \] \cite{green}, \cite{ringel}.

A natural question to ask is whether there is some kind of enlarged version of the Hall algebra from which one could recover not just $U_q(\mathfrak b^+)$, but all of $U_q(\mathfrak g)$.  Work of Peng and Xiao \cite{px} has led to the conjecture that such an algebra should be obtained from the following category.  Using the abelian category $\mathcal A$ of quiver representations as above, consider its bounded derived category $\mathcal D^\flat(\mathcal A)$, which is no longer abelian, but is instead a triangulated category.  As such, it has a shift functor $\Sigma \colon \mathcal D^\flat (\mathcal A) \rightarrow \mathcal D^\flat (\mathcal A)$.  We then define the \emph{root category} of $\mathcal A$ to be $\mathcal D^\flat (\mathcal A)/\Sigma^2$, the triangulated category obtained from $\mathcal D^\flat (\mathcal A)$ by identifying an object with its double shift.

It is still an open question how to find a ``Hall algebra" associated to this root category.  To begin with, the usual definition does not apply because the root category is not abelian.  It is, however, triangulated, and recent efforts in this area have focused on finding Hall algebras for triangulated categories.  In the rest of this section, we describe derived Hall algebras, defined by To\"en, which can be obtained from certain triangulated categories.  Thus far the necessary restrictions on these triangulated categories prohibit us from being able to define a derived Hall algebra for the root category.

\subsection{Derived Hall algebras}

Recall that a \emph{differential graded category} or \emph{dg category}, is a category enriched over $Ch(R)$, the category of cochain complexes of modules over a ring $R$.  Thus, given any objects $x$ and $y$ in a dg category $\mathcal T$, we have a cochain complex $\mathcal T(x,y)$.  Here, we assume that $R=\mathbb F_q$, the finite field with $q$ elements.  To\"en defines a dg category $\mathcal T$ to be \emph{locally finite} if for any objects $x$ and $y$ in $\mathcal T$, the cochain complex $\mathcal T(x,y)$ is cohomologically bounded and has all cohomology groups finite dimensional \cite[3.1]{toendha}.

Given a locally finite dg category $\mathcal T$, we consider $\mathcal M(\mathcal T)$, the category of dg $\mathcal T^{op}$-modules, or functors $\mathcal T \rightarrow Ch(\mathbb F_q)$.  This category has the structure of a stable model category, with levelwise weak equivalences and fibrations \cite[\S 3]{toendg}.  We have made finiteness assumptions about the dg category $\mathcal T$, but in taking the module category, we may have cochain complexes in the image which do not satisfy these kinds of conditions.  If we restrict to functors which are appropriately finitary, we no longer have a model structure, since this subcategory does not possess enough limits and colimits.  So, we work with the model category $\mathcal M(\mathcal T)$ of all modules but consider also the full subcategory $\mathcal P(\mathcal T)$ of perfect objects.  A module in $\mathcal M(\mathcal T)$ is \emph{perfect} if it belongs to the smallest subcategory of $\Ho(\mathcal M(\mathcal T))$ containing the quasi-representable modules (see \cite[3.6]{toendg} for a definition) and which is stable by retracts, homotopy pushouts, and homotopy pullbacks \cite{toendha}.  Perfect objects coincide with the compact objects in the triangulated category $\Ho(\mathcal M(\mathcal T))$.  (Recall that if $T$ is a triangulated category with arbitrary coproducts, then an object $x$ of $T$ is \emph{compact} if any map $x \rightarrow \amalg_i y_i$ factors through a finite coproduct \cite[6.5]{krause}.)

Since $\Ho \mathcal M(\mathcal T)$ is a triangulated category, it has a shift functor; we denote maps from $x$ to the $i$th shift of $y$ in this category by $[x, y[i]]$ or by $\Ext^i(x,y)$.  Notice that for perfect modules, these $\Ext$ groups are all finite.

\begin{theorem} \cite[1.1, 5.1]{toendha}
Let $\mathcal T$ be a locally finite dg category over a finite field $\mathbb F_q$.  Define $\mathcal{DH}(\mathcal T)$ to be the $\mathbb Q$-vector space with basis the characteristic functions $\chi_x$, where $x$ runs through the set of weak equivalence classes of perfect objects in $\mathcal M(\mathcal T)$.  Then there exists an associative and unital product
\[ \mu \colon \mathcal{DH}(\mathcal T) \otimes \mathcal{DH}(\mathcal T) \rightarrow \mathcal{DH}(\mathcal T) \] such that
\[ \mu(\chi_x, \chi_y)= \sum_z g^z_{x,y} \chi_z \] and these derived Hall numbers $g^z_{x,y}$ are given by the formula
\[ g^z_{x,y}= \frac{|[x,z]_y| \cdotp \prod_{i>0} |\Ext^{-i}(x,z)|^{(-1)^i}}{|\Aut(x)| \cdotp \prod_{i>0} |\Ext^{-i}(x,x)|^{(-1)^i}}, \] where $[x,z]_y$ denotes the subset of $[x,z]$ of morphisms $f \colon x \rightarrow z$ whose cone is isomorphic to $y$ in $\Ho(\mathcal M(\mathcal T))$.
\end{theorem}

\section{More general derived Hall algebras}

In this section, we establish the existence of derived Hall algebras for sufficiently finitary stable complete Segal spaces.  Our strategy follows that of To\"en, and some proofs of his continue to hold without change.  However, without the restrictions of a model structure, some of the proofs are greatly simplified.

Throughout this section, suppose that $W$ is a pointed stable complete Segal space, so that $\Ho(W)$ is a triangulated category with a zero object.  As in the previous section, we define for any objects $x, y$ in $W$
\[ \Ext^i(x,y)= [x,y[i]] \] where the outside brackets denote maps in $\Ho(W)$ and the inside brackets denote the shift functor giving the triangulated structure of $\Ho(W)$.

\begin{definition}
A stable complete Segal space $W$ is \emph{finitary} if in $\Ho(W)$ we have that $\Ext^i(x,y)$ is finite for all pairs of objects $(x,y)$ and all values of $i$, and zero for sufficiently large values of $i$.
\end{definition}

We assume for the rest of the paper that all our stable complete Segal spaces are finitary.

Since the model category $\css$ is cartesian closed, the simplicial space $W^{\Delta[1]}$ is also a complete Segal space.  Notice that $W$ itself is isomorphic to the mapping object $W^{\Delta[0]}$, and so we can use the two maps $\Delta[0] \rightarrow \Delta[1]$ to define ``source" and ``target"  maps $s,t \colon W^{\Delta[1]} \rightarrow W$.  Since an object of $W^{\Delta[1]}$ is a 0-simplex $u \in \map_W(x,y)$ for some $x$ and $y$ objects of $W$, these two maps can be defined by $s(u)=x$ and $t(u)=y$.   We also have a ``cone" map $c \colon W^{\Delta[1]} \rightarrow W$ given by $c(u)= y \amalg_x 0$, where such a cone object exists because we have required that $W$ be stable; in the homotopy category, it is just the completion of $u \colon x \rightarrow y$ to a distingushed triangle.

Using these maps, we can put together the diagram
\[ \xymatrix{W^{\Delta[1]} \ar[r]^t \ar[d]_{s \times c} & W \\
W \times W & } \] analogous to To\"en's diagram of model categories \cite[\S 4]{toendha}.

Because we are no longer working with model categories, a number of aspects of this diagram have been simplified, compared to the analogous one in To\"en's paper.  Because the objects are complete Segal spaces, rather than model categories, we no longer have to be concerned with whether these maps are left Quillen functors.  Furthermore, we are able to impose conditions on $W$ from the beginning so that its objects are already ``perfect" in that all the necessary finiteness conditions are already satisfied.

A word on this point would perhaps be helpful here.  It is likely that a stable complete Segal space that would arise in nature would not have all pairs of objects $x$ and $y$ satisfying the necessary finiteness conditions on $\Ext^i(x,y)$.  However, we can show that restricting to the sub-complete Segal space with objects satisfying such conditions is still a complete Segal space.  Explicitly, given a complete Segal space $W$, consider the doubly constant simplicial space $W_{0,0}$, and the sub-simplicial space $Z_{0,0}$ given by the perfect objects of $W$.  Then define $Z$ to be the simplicial spaces given by the pullback
\[ \xymatrix{Z \ar[r] \ar[d] & W \ar[d] \\
Z_{0,0} \ar[r] & W_{0,0}.} \]
Since $W_{0,0}$ is discrete, the map $W \rightarrow W_{0,0}$ is a fibration in $\css$, from which it follows that the map $Z \rightarrow Z_{0,0}$ is a fibration also.  Thus, $Z$ is a fibrant simplicial space in $\css$, or a complete Segal space.  Furthermore, since compact objects of a triangulated category form a triangulated subcategory \cite[6.5]{krause}, $\Ho(Z)$ is triangulated and $Z$ is stable.  Thus, we can restrict to the appropriate setting without losing the structure that we need, and so we always assume that, given an arbitrary stable complete Segal space $W$, we have implicitly restricted to $Z$.

Now, as To\"en does, we restrict to the sub-complete Segal spaces of $W$ and $W^{\Delta[1]}$, whose mapping spaces are sent to isomorphisms in the homotopy category; we call these spaces $wW$ and $wW^{\Delta[1]}$, respectively.  Taking the nerve of the homotopy categories, we obtain a diagram
\[ \xymatrix{\nerve(\Ho(wW^{\Delta[1]})) \ar[r]^t \ar[d]_{s \times c} & \nerve(\Ho(wW)) \\
\nerve(\Ho(wW)) \times \nerve(\Ho(wW)). } \]
For simplicity of notation, we write this diagram
\[ \xymatrix{X^{(1)} \ar[r]^t \ar[d]_{s \times c} & X^{(0)} \\
X^{(0)} \times X^{(0)}. } \]

To get an algebra with a well-defined multiplication, we need to show that this diagram of spaces satisfies some properties.

\begin{definition} \cite[2.1]{toendha}
An object $X$ in the homotopy category of spaces is \emph{locally finite} if it satisfies the conditions
\begin{enumerate}
\item for any base point $x \in X$ and $i >0$, the group $\pi_i(X,x)$ is finite, and

\item for any base point $x \in X$, there is some $n$, depending on $x$, such that $\pi_i(X,x)=0$ for all $i>n$.
\end{enumerate}
\end{definition}

\begin{lemma}
The spaces $X^{(0)}$ and $X^{(1)}$ are locally finite.
\end{lemma}

\begin{proof}
For any $x \in \pi_0(X^{(0)})$, we use the facts that
\[ \pi_1(X^{(0)}) \subseteq \Ext^0(x,x)=[x,x] \] and
\[ \pi_i(X^{(0)}) = \Ext^{1-i}(x,x) \] for $i>1$ \cite[3.2]{toendha}.  Our assumption on $W$ guarantees that these groups are all finite, and that they are zero for sufficiently large $i$.  Thus, $X^{(0)}$ is locally finite.

To show that $X^{(1)}=\nerve(\Ho(wW^{\Delta[1]}))$ is locally finite, notice that this space is weakly equivalent to
\[ \nerve(\Ho(wW)) \times \Delta[1] = X^{(0)} \times \Delta [1] \] which is also locally finite.
\end{proof}

\begin{definition} \cite[2.5]{toendha}
A morphism $f \colon X \rightarrow Y$ of locally finite homotopy types is \emph{proper} if, for any $y \in \pi_0(Y)$, there are only finitely many $x \in \pi_0(X)$ with $f(x)=y$.
\end{definition}

Notice that $f$ is proper if and only if, for any $y \in \pi_0(Y)$, the set $\pi_0(F_y)$ is finite.  The proof of the following lemma follows just as it does in To\"en's paper \cite[3.2]{toendha}.

\begin{lemma}
The map $s \times c$ is proper.
\end{lemma}

With these properties established for our diagram, we can use it to define an algebra analogous to that of To\"en \cite[\S 4]{toendha}.

\begin{definition} \cite[2.2]{toendha}
Let $X$ be a space.  The $\mathbb Q$-vector space of \emph{rational functions with finite support} on $X$ is the $\mathbb Q$-vector space of functions on the set $\pi_0(X)$ with values in $\mathbb Q$ and finite support, and is denoted by $\mathbb Q_c(X)$.
\end{definition}

\begin{definition}
As a vector space, the \emph{derived Hall algebra} $\dhw$ of $W$ is given by $\mathbb Q_c(X^{(0)})$.
\end{definition}

Given a morphism $f \colon X \rightarrow Y$ of locally finite spaces, we define a push-forward morphism $f_! \colon \mathbb Q_c(X) \rightarrow \mathbb Q_c(Y)$ as follows.  Given $y \in \pi_0(Y)$, let $F_y$ denote the homotopy fiber of $f$ over $y$, and let $i \colon F_y \rightarrow X$ be the natural map.  Using the long exact sequences of homotopy groups, one can see that for any $z \in \pi_0(F_y)$, the group $\pi_i(F_y,z)$ is finite for all $i>0$ and zero for sufficiently large $i$.  Furthermore, the fibers of the map $\pi_0(F_y) \rightarrow \pi_0(X)$ are all finite.  Then, for any $\alpha \in \mathbb Q_c(X)$ and $y \in \pi_0(Y)$, define the function $f_!$ by
\[ f_!(\alpha)(y)= \sum_{z \in \pi_0(F_y)} \alpha(i(z)) \cdotp \prod_{i>0}|\pi_i(F_y, z)|^{(-1)^i}. \] The assumption that $\alpha$ have finite support guarantees that $f_!$ is well-defined.

%\begin{lemma} \cite[2.3]{toendha}
%Given any $f \colon X \rightarrow Y$, $\alpha \in \mathbb Q_c(X)$, and $y \in \pi_0(Y)$, the map $f_!$ can alternatively be given as
%\[ f_!(\alpha)(y) = \sum_{x \in \pi_0(x), f(x)=y} \alpha(x) \cdotp \prod_{i>0} \frac{|\pi_i(X,x)|^{(-1)^i}}{|\pi_i(Y,y)|^{(-1)^i}}. \]
%\end{lemma}
%
%\begin{cor} \cite[2.4]{toendha}
%Given two morphisms $f \colon X \rightarrow Y$ and $g \colon Y \rightarrow Z$ of locally finite homotopy types, $(g \circ f)_!= g_! \circ f_!$.
%\end{cor}

If $f \colon X \rightarrow Y$ is a proper map of locally finite spaces, then we have a well-defined pullback $f^* \colon \mathbb Q_c(Y) \rightarrow \mathbb Q_c(X)$ defined in the usual way as $f^*(\alpha)(x)= \alpha(f(x))$ for any $\alpha \in \mathbb Q_c(Y)$ and $x \in \pi_0(X)$.  The requirement that $f$ be proper guarantees that $f^*(\alpha)$ has finite support, so that $f^*$ is in fact well-defined.

The following lemma is key for establishing associativity.

\begin{lemma} \cite[2.6]{toendha} \label{cartesian}
Consider a homotopy pullback diagram of locally finite spaces
\[ \xymatrix{X' \ar[r]^{v} \ar[d]_g & X \ar[d]^f \\
Y' \ar[r]^u & Y} \] with $u$ proper.  Then the map $v$ is also proper, and
\[ u^* \circ f_! = g_! \circ v^* \colon \mathbb Q_c(X) \rightarrow \mathbb Q_c(Y'). \]
\end{lemma}

To define the multiplication on $\dhw$, first notice that we have an isomorphism
\[ \dhw \otimes \dhw \rightarrow \mathbb Q_c(X^{(0)} \times X^{(0)}) \] given by
\[ (f,g) \mapsto ((x,y) \mapsto f(x) \cdotp g(x)). \]  Then we can consider the map
\[ \mu = t_! \circ (s \times c)^* \colon \dhw \otimes \dhw \rightarrow \dhw. \]

The algebra structure on $\dhw$ is then given by
\[ x \cdotp y = \sum_z g^z_{x,y} z \] where
\[ g^z_{x,y}= \mu(\chi_x, \chi_y)(z) \] where $\chi_x$ denotes the characteristic function of $x$.

\begin{prop}
With this multiplication, $\dhw$ is a unital algebra.
\end{prop}

Our proof essentially follows the one given by To\"en \cite[4.1]{toendha}, with the necessary changes being made as we translate to the complete Segal space setting.

\begin{proof}
Given any object $x$ in $W$, let $\chi_x$ denote its characteristic function; in particular, consider $\chi_0$, the characteristic function of the zero object of $W$.

Notice that the set $\pi_0(X^{(1)})$ is isomorphic to the set of isomorphism classes of objects in $\Ho(wW^{\Delta[1]})$.  Thus, fix some 0-simplex $u \colon x \rightarrow y$ of $\map_W(x,y)$, regarded as an object of $\Ho(wW^{\Delta[1]})$.  Then
\[ (s \times c)^*(u) =
\begin{cases}
1 & \text{if } y \cong 0 \text{ and } x \cong z \text{ in } \Ho(wW) \\
0 & \text{otherwise.}
\end{cases} \]
In other words, $(s \times c)^*(\chi_0, \chi_x)$ is the characteristic function of the subset of $\pi_0(X^{(1)})$ consisting of maps $0 \rightarrow z$ with $z \cong x$ in $\Ho(wW)$.

Define $X$ to be the simplicial set contained in $X^{(1)}$ consisting of all the support of $(s \times c)^*(\chi_0, \chi_x)$, and notice that $X$ is a connected simplicial set.  Then using the definition of the product map $\mu$, we get
\[ \mu(\chi_0, \chi_x)(x) = \prod_{i>0} \left( |\pi_i(X)|^{(-1)^i} \cdotp |\pi_i(X^{(0)},x)|^{(-1)^{i+1}} \right). \]  Notice in particular that whenever $y \neq x$,
\[ \mu(\chi_0, \chi_x)(y)=0. \]

Restricting the target map $t \colon W^{\Delta[1]} \rightarrow W$ to the maps $y \rightarrow z$ such that $y \cong 0$ in $\Ho(wW)$, we see that on such objects $t$ is fully faithful, up to homotopy.  Thus, the induced map $t \colon X \rightarrow X^{(0)}$ induces isomorphisms $t_* \colon \pi_i(X) \rightarrow \pi_i(X^{(0)})$ for all $i>0$, and the simplicial set $X$ can be identified with a connected component of $X^{(0)}$.  Hence, $\mu(\chi_0, \chi_x)(x)=1$, so that $\mu(\chi_0, \chi_x) = \chi_x$.

Changing the order and following the same argument, one can see that we also have $\mu(\chi_x, \chi_0)= \chi_x$, thus proving that $\chi_0$ is a unit element for $\dhw$.
\end{proof}

\begin{theorem}
With this multiplication, $\dhw$ is an associative algebra.
\end{theorem}

\begin{proof}
Consider the complete Segal space $W^{\Delta[2]}$, and, as with $W^{\Delta[1]}$ and $W$, denote by $X^{(2)}$ the simplicial set $\nerve(\Ho(wW^{\Delta[2]}))$.  Notice that there are three natural maps \[ f,g,h \colon W^{\Delta[2]} \rightarrow W^{\Delta[1]} \] induced by the three inclusion maps $\Delta[1] \rightarrow \Delta[2]$, where $f$ sends $x \rightarrow y \rightarrow z$ to $x \rightarrow y$, $g$ sends it to $y \rightarrow z$, and $h$ sends it to $x \rightarrow z$.  There is also a cone map
\[ k \colon W^{\Delta[2]} \rightarrow W^{\Delta[1]} \] given by
\[ (x \rightarrow y \rightarrow z) \mapsto (y \amalg_x 0 \rightarrow z \amalg_x 0), \] with the pushouts defined as before in a stable complete Segal space, and a map between the two given by the universal property.  This map may not be unique, but all such maps form a weakly contractible space.

Using these maps, we get two diagrams:
\[ \xymatrix{X^{(2)} \ar[r]^g \ar[d]_{f \times (c \circ k)} & X^{(1)} \ar[r]^t \ar[d]^{s \times c} & X^{(0)} \\
X^{(1)} \times X^{(0)} \ar[r]^{t \times \id} \ar[d]_{(s \times c) \times \id} & X^{(0)} \times X^{(0)} & \\
(X^{(0)} \times X^{(0)}) \times X^{(0)} && } \] and
\[ \xymatrix{X^{(2)} \ar[r]^h \ar[d]_{(s \circ f) \times k} & X^{(1)} \ar[r]^t \ar[d]^{s \times c} & X^{(0)} \\
X^{(0)} \times X^{(1)} \ar[r]^{\id \times t} \ar[d]_{\id \times (s \times c)} & X^{(0)} \times X^{(0)} & \\
X^{(0)} \times (X^{(0)} \times X^{(0)}) && } \] which both give the same result taking composites across the top and down the left side:
\[ \xymatrix{X^{(2)} \ar[r] \ar[d] & X^{(0)} \\
X^{(0)} \times X^{(0)} \times X^{(0)}. & } \]

Thus, to prove associativity of $\dhw$, it suffices by Lemma \ref{cartesian} to prove that the square in each of these diagrams is homotopy cartesian.
In fact, it suffices to show that the diagrams
\[ \xymatrix{X^{(2)} \ar[r]^g \ar[d]_f & X^{(1)} \ar[d]^s \\
X^{(1)} \ar[r]^t & X^{(0)} } \hskip .5 in
\xymatrix{X^{(2)} \ar[r]^h \ar[d]^k & X^{(1)} \ar[d]^c \\
X^{(1)} \ar[r]^t & X^{(0)} } \] are homotopy cartesian.  For the first diagram, this fact follows immediately from the fact that the original diagram
\[ \xymatrix{ W^{\Delta[2]} \ar[r]^g \ar[d]_f & W^{\Delta[1]} \ar[d]^s \\
W^{\Delta[1]} \ar[r]^t & W } \] is a homotopy pullback diagram of complete Segal spaces.   To show that the second diagram is homotopy cartesian requires more effort.

In this second diagram, let $Z$ denote the homotopy pullback $W^{\Delta[1]} \times^h_W W^{\Delta[1]}$.  Using Lemma \ref{lemma}, it suffices to prove that $\Ho(W^{\Delta[2]}) \rightarrow \Ho(Z)$ is fully faithful and essentially surjective.  We begin with the argument for the latter.  Suppose we have an object $(x \rightarrow z, w \rightarrow z \amalg_x 0)$ in $\Ho(Z)$; we want to find an object $y$ of $W$ such that $x\rightarrow y \rightarrow z$ is an object of $\Ho(W^{\Delta[2]})$ with $y \amalg_x 0 = w$.  Such a $y$ can be found by applying the axioms for a triangulated category to the diagram
\[ \xymatrix{x \ar@{-->}[r] \ar[d]_= & y \ar@{-->}[r] \ar@{-->}[d] & w \ar[r] \ar[d] & x[1] \ar[d]^= \\
x \ar[r] & z \ar[r] & z \amalg_x 0 \ar[r] & x[1].} \]

To prove that the functor is fully faithful, we need to prove that, for any objects $x \rightarrow y \rightarrow z$ and $x' \rightarrow y' \rightarrow z'$ in $\Ho(W^{\Delta[2]})$, the map
 \begin{multline*}
 \Hom_{\Ho(W^{\Delta[2]})}(x \rightarrow y \rightarrow z, x' \rightarrow y' \rightarrow z') \rightarrow \\ \Hom_{\Ho(Z)}((x \rightarrow z, y \amalg_x 0 \rightarrow z \amalg_x 0), (x' \rightarrow z', y' \amalg_{x'} 0 \rightarrow z' \amalg_{x'} 0))
\end{multline*}
is an isomorphism.
Elements of the set on the left-hand side are triples of maps making the diagram
\[ \xymatrix{x \ar[r] \ar[d] & y \ar[r] \ar[d] & z \ar[d] \\
x' \ar[r] & y' \ar[r] & z'} \] commute, where elements of the set on the right-hand side are 4-tuples of maps making the pair of diagrams
\[ \xymatrix{x \ar[r] \ar[d] & z \ar[d] \\
x' \ar[r] & z',} \hskip .5 in
\xymatrix{y \amalg_x 0 \ar[r] \ar[d] & z \amalg_x 0 \ar[d] \\
y' \amalg_{x'} 0 \ar[r] & z' \amalg_{x'} 0} \] commute.  Given an element of the right-hand set, we can use the axioms for a triangulated category to find a map $y \rightarrow y'$ compatible with the maps $x \rightarrow x'$ and $z \rightarrow z'$ to obtain an element of the left-hand set.  Thus, the map is surjective.  A similar argument can be used to prove that it is injective.
\end{proof}

The proof of the following formula is essentially the same as the one given by To\"en \cite[5.1]{toendha}; we give it here with the necessary changes to our situation.

\begin{prop}
The derived Hall numbers are given by
\[ g^z_{x,y}= \frac{|[x,z]_y| \cdotp \prod_{i>0} |\Ext^{-i}(x,z)|^{(-1)^i}}{|\Aut(x)| \cdotp \prod_{i>0} |\Ext^{-i}(x,x)|^{(-1)^i}}, \] where $[x,z]_y$ denotes the subset of $[x,z]$ of morphisms $f \colon x \rightarrow z$ whose cone is isomorphic to $y$ in $\Ho(W)$.
\end{prop}

\begin{proof}

Given the target map $t \colon X^{(1)} \rightarrow X^{(0)}$ and an object $z$ of $Ho(W)$, let $F^z$ denote the homotopy fiber of $t$ over $z$.  Using the definitions of $X^{(1)}$ and $X^{(0)}$, notice that $F^z$ is weakly equivalent to the nerve of the category $\text{equiv}(W \downarrow z)$ whose objects are maps from arbitrary objects of $W$ to $z$, and whose morphisms are the homotopy equivalences of $W$, making the resulting triangular diagram commute.

Given two other objects $x$ and $y$ of $W$, let $F^z_{x,y}$ denote the nerve of the full subcategory of $\text{equiv}(W \downarrow z)$ whose objects are the maps $u \colon x' \rightarrow z$, where $x' \simeq x$, and whose cofiber is equivalent to $y$.  Notice that $F^z_{x,y}$ is locally finite, since both $X^{(1)}$ and $X^{(0)}$ are; moreover, $\pi_0(F^z_{x,y})$ is finite, and it is isomorphic to $[x,z]_y/\Aut(x)$.

Using $F^z_{x,y}$, we can reformulate our definition of the derived Hall number $g^z_{x,y}$ as
\[ g^z_{x,y}= \sum_{(u \colon x' \rightarrow y) \in \pi_0(F^z_{x,y})} \prod_{i>0} |\pi_i(F^z_{x,y}, u)|^{(-1)^i}. \]

We first prove that
\[ \prod_{i>0} |\pi_i(F^z_{x,y},u)|^{(-1)^i}= |\Aut(f/z)|^{-1} \cdotp \prod_{i>0} |\Ext^{-i}(x,z)|^{(-1)^i} \cdotp |\Ext^{-i}(x,x)|^{(-1)^{i+1}}, \] where $\Aut(f/z)$ denotes the stabilizer of a map $f \in [x,z]_y$ under the action of $\Aut(x)$.

Notice that we get a homotopy cartesian square of mapping spaces
\[ \xymatrix{\map_{W \downarrow z}(x,x) \ar[r] \ar[d] & \map_W(x,x) \ar[d] \\
\ast \ar[r] & \map_W(x,z) } \] where the bottom horizontal map specifies the map $u \colon x \rightarrow z$.  Thus, we have a fibration of simplicial sets, and hence a long exact sequence of homotopy groups
\begin{multline*}
\cdots \rightarrow \pi_2(\map(x,z)) \rightarrow \pi_1(\map_{W \downarrow z}(x,x)) \rightarrow \pi_1(\map_W(x,x)) \rightarrow \pi_1(\map_W(x,z)) \\ \rightarrow \pi_0(\map_{W \downarrow z}(x,x)) \rightarrow \pi_0(\map_W(x,x)) \rightarrow \pi_0(\map_W(x,z)) \rightarrow 0.
\end{multline*}
Composing the last two maps between nontrivial sets, we get a surjection
\[ \pi_0(\map_{W \downarrow z}(x,x)) \rightarrow \Aut(f/z). \]

Furthermore, notice that $\pi_i(\map_W(x,z))= [x, z[-i]]= \Ext^{-i}(x,z)$ and, similarly, that $\pi_i(\map_W(x,x))= \Ext^{-i}(x,x)$.  Finally, observe that $\pi_i(\map_{W\downarrow z}(x,x))$ is weakly equivalent to $\pi_{i+1}(\nerve(\text{equiv}(W \downarrow z)), u)$, which, as we have noted previously, is equivalent to $\pi_{i+1}(F^z_{x,y}, u)$.  Thus, we have a long exact sequence
\begin{multline*}
\cdots \rightarrow \Ext^{-2}(x,z) \rightarrow \pi_2(F^z_{x,y},u) \rightarrow \Ext^{-1}(x,x) \rightarrow \Ext^{-1}(x,z) \\ \rightarrow \pi_1(F^z_{z,y},u) \rightarrow \Aut(f/z) \rightarrow 0.
\end{multline*}
Using properties of long exact sequences, we obtain the equation given above.

To prove the statement of the proposition, we use the fact that, since $\Aut(x)$ is a finite group and $[x,z]_y$ is a finite set, we get that
\[ \frac{|[x,z]_y|}{|\Aut(x)|} = \sum_{f \in ([x,z]_y/\Aut(x))} |\Aut(f/x)|. \]  The formula follows.
\end{proof}

\end{document}